\documentclass[11pt,american,english]{article}
\usepackage[T1]{fontenc}
\usepackage[latin9]{inputenc}
\usepackage[a4paper]{geometry}
\usepackage{babel}
\usepackage{amsmath}
\usepackage{amsthm}
\usepackage{amssymb}
\usepackage{setspace}
\usepackage[unicode=true,pdfusetitle,
 bookmarks=true,bookmarksnumbered=true,bookmarksopen=true,bookmarksopenlevel=3,
 breaklinks=false,pdfborder={0 0 1},backref=false,colorlinks=false]
 {hyperref}

\makeatletter
\theoremstyle{plain}
\newtheorem{thm}{\protect\theoremname}
  \theoremstyle{remark}
  \newtheorem{rem}[thm]{\protect\remarkname}
  \theoremstyle{plain}
  \newtheorem{prop}[thm]{\protect\propositionname}

\usepackage{multirow}

\usepackage{arydshln}

\usepackage{lmodern}
\newcommand{\1}{\mbox{1\hspace{-1mm}I}}
\numberwithin{equation}{section}

\usepackage{calc}
\makeatother

  \addto\captionsamerican{\renewcommand{\propositionname}{Proposition}}
  \addto\captionsamerican{\renewcommand{\remarkname}{Remark}}
  \addto\captionsamerican{\renewcommand{\theoremname}{Theorem}}
  \addto\captionsenglish{\renewcommand{\propositionname}{Proposition}}
  \addto\captionsenglish{\renewcommand{\remarkname}{Remark}}
  \addto\captionsenglish{\renewcommand{\theoremname}{Theorem}}
  \providecommand{\propositionname}{Proposition}
  \providecommand{\remarkname}{Remark}
\providecommand{\theoremname}{Theorem}

\begin{document}
\selectlanguage{american}%
\global\long\def\1{\mbox{1\hspace{-1mm}I}}
\selectlanguage{english}%

\title{MEAN FIELD APPROACH TO STOCHASTIC CONTROL WITH PARTIAL INFORMATION}

\author{Alain Bensoussan\\
\thanks{The first author acknowledges the financial support by the National Science Foundation DMS- 1612880, and the
Research Grants Council of the Hong Kong Special Administrative Region
(CityU , CityU 113 03 316). The second author acknowledges the financial support from The Hong Kong RGC GRF 14301015 with the project title: Advance in Mean Field Theory and The Hong Kong RGC GRF 14300717 with the project title: New kinds of forward-backward stochastic systems with applications, and the financial support from the Faculty of Science of Chinese University of Hong Kong via the CUHK Direct Grants with internal code numbers: 3132761 and 3132762.}International Center for Decision and Risk Analysis\\ Jindal School of Management, University of Texas at Dallas\\
School of Data Science, City University Hong Kong\\
 \\
Sheung Chi Phillip Yam\\
Department of Statistics, The Chinese University of Hong Kong }
\maketitle

\section{INTRODUCTION }

The classical stochastic control problem under partial information
, as , for instance , described in the book of A. Bensoussan\cite{ABE}
, can be formulated as a control problem for Zakai equation, whose
solution is the unnormalized conditional probability distribution
of the state of the system, which is not directly accessible. Zakai
equation is a stochastic Fokker-Planck equation. Therefore, the mathematical
problem to be solved is very similar to that met in Mean Field Control
theory. Since Mean Field Control theory is much posterior to the development
of Stochastic Control with partial information, the tools, techniques
and concepts obtained in the last decade, for Mean Field Games and
Mean field type Control theory, have not been used for the control
of Zakai equation. It is the objective of this work to conncet the
two theories. Not only , we get the power of new tools, but also we
get new insights for the problem of stochastic control with partial
information. For mean field theory, we get new interesting applications,
but also new problems. The possibility of using direct methods is
, of course, quite fruitful. Indeed, if Mean Field Control Theory
is a a very comprehensive and powerful framework , it leads to very
complex rquations, like the Master equation, which is a nonlinear
infinite dimensional P.D.E., for which general theorems are hardly
available, although an active research in this direction is performed,
see P. Cardialaguet, F. Delarue, J.M. Lasry, P.L. Lions \cite{CDLL}.
Direct methods are particularly useful to obtain regularity results.
We will develop in detail the linear quadratic regulator problem,
but because we cannot just consider the gaussian case , well know
results , like the separation principle are not available. An interesting
and important result is available in the literature, due to A. Makowsky,
\cite{ARM}. It describes the solution of Zakai equation for linear
systems with general initial condition ( non-gaussian). Curiouly ,
this result had not been exploited for the control aspect, in the
literature. We show that the separation principle can be extended
for quadratic pay-off functionals, but the Kalman filter is much more
complex than in the gaussian case. Finally we compare our work to the work of Bandini, Corso, Fuhrman and Pham \cite{BCFP} and we show that the example E. Bandini et al. provided does not cover ours. Our system remains nonlinear in their setting.

\section{STOCHASTIC CONTROL WITH PARTIAL INFORMATION}

\subsection{THE PROBLEM }

We describe the problem formally, without making precise the assumptions.
The state of the system $x(t)\in R^{n}$ is solution of a diffusion 

\begin{equation}
dx=g(x,v)dt+\sigma(x)dw\label{eq:2-1}
\end{equation}

\[
x(0)=\xi
\]
so, we assume that there exists a probability space $\Omega,\mathcal{A},P$
on which are constructed a random variable $\xi$ and a standard Wierner
process in $R^{n}$, which is independent of $\xi$ . There is a control
$v(t)$ in the drift term , with values in $R^{m}$. Since, we cannot
access to the state $x(t),$ which is not observable, it cannot be
defined by a feedback on the state, nor adapted to the state. Formally,
we have an observation equation 

\begin{equation}
dz=h(x)dt+db(t)\label{eq:2-2}
\end{equation}
in which $z(t)$, with values in $R^{d}$ , represents the observation
and $b(t)$ is also a Wiener process, independent of the pair $(\xi,w(.)).$
The function $h(x)$ corresponds to the measurement of the state $x$
and $b(t)$ captures a measurement error. So the control $v(t)$ should
be adapted to the process $z(t),$not a feedback of course. It is
well known that this construction is ill-posed. Indeed, the control
is adapted to the observation, which depends also on the state, which
depends on the control. It is a chicken and egg effect, that is usually
solved by the Girsanov theorem, at the price of constructing apropriately
the Wiener process $b(t).$ In practice, we construct on $($$\Omega,\mathcal{A},P$)
three objects, $\xi,w(.),z(.).$ The processes $w(.),z(.)$ are independent
Wiener processes on $R^{n},$$R^{d}$ respectively and $\xi$ is independent
of these two processes. We set 

\[
\mathcal{F}^{t}=\sigma(\xi,w(s),z(s),s\leq t)\quad\mathcal{Z}^{t}=\sigma(z(s),s\leq t)
\]
 the filtrations on $(\Omega,\mathcal{A},P)$ generated by $(\xi,w(.),z(.))$
and $z(.)$ respectively. The process $z(.)$ is the observation process,
but it is defined externally. We can then choose the control $v(.)$
as a process with values in $R^{m},$which is adapted to the filtration
$\mathcal{Z}^{t}.$ So , it is perfectly well defined. , as well as
the process $x(.)$ solution of (\ref{eq:2-1}). In fact in (\ref{eq:2-1})
$v(.)$ is fixed, like $\xi$ and $w(.),$and we assume that we can
solve the S.D.E. (\ref{eq:2-1}) in a strong sense. So $x(.)$ is
well defined. Here comes Girsanov theorem. We define the scalar $P,\mathcal{F}^{t}$
martingale $\eta(t)$, solution of the equation 

\begin{equation}
d\eta(t)=\eta(t)\,h(x(t)).dz(t),\:\eta(0)=1\label{eq:2-3}
\end{equation}
This martingale allows to define a mew probability on $\Omega,\mathcal{A}$
, denoted $P^{v(.)}$ to emphasize the fact that it depends on the
control $v(.).$ It is given by the Radon- Nikodym derivative 

\begin{equation}
\dfrac{dP^{v(.)}}{dP}|_{\mathcal{F}^{t}}=\eta(t)\label{eq:2-4}
\end{equation}
Finally, we define the process 

\begin{equation}
b^{v(.)}(t)=z(t)-\int_{0}^{t}h(x(s))ds\label{eq:2-5}
\end{equation}
which also depends on the control decision . We take a finite horizon
$T$, to fix ideas. Making the change of probability from $P$ to
$P^{v(.)}$ and considering the probability space $(\Omega,\mathcal{F}^{T},P^{v(.)})$
, then $b^{v(.)}$ appears as a standard Wiener process, which is
independent of $w(.)$ and $\xi.$ Therefore, (\ref{eq:2-5}) is a
template of (\ref{eq:2-2}) as far as probability laws are concerned.
We can then rigorously define the control problem ( without the chicken
and egg effect) 

\begin{equation}
J(v(.))=E^{v(.)}[\int_{0}^{T}f(x(t),v(t))dt+f_{T}(x(T))]\label{eq:2-6}
\end{equation}
in which the functions $f(x,v)$ and $f_{T}(x)$ represent the running
cost and the final cost contributing to the pay off functional to
be minimized. The notation $E^{v(.)}$ refers to the expected value
with respect to the probability law $P^{v(.)}.$ 
\begin{rem}
\label{rem2-1} The previous presentation , which is currently the
common one to formalize stochastic control problems with partial information,
has a slight drawback, in comparison with the description of the problem
with full information. With full information , there is no $\mathcal{Z}^{t}$
and the underlying filtration $\mathcal{F}^{t}=\sigma(\xi,w(.))$
is accessible. A control $v(.)$ is a stochastic process adapted to
$\mathcal{F}^{t}$. We call it open-loop , because it is externally
defined ( this should not be confused with the practice in engineering
to call open -loop controls, those which are deterministic functions
of time). But , since the state $x(t)$ is also accessible , we can
also consider controls , defined by feedbacks built on the state.
In spite of the difference in the definition, the class of feedback
controls is contained in that of open-loop controls. Indeed , after
constructing the trajectory corresponding to a feedback, we feed the
feedback with that trajectory. We get an open-loop control, leading
to the same cost. The interesting feature of cost functionals of the
type (\ref{eq:2-6}) is that the optimal open-loop control is defined
by a feedback. So restricting ourselves to the subclass of feedback
controls does not hurt. This is very important, when we formulate
the control problem in the framework of mean-field theory. In mean
-field theory , we must define the control with a feedback. Surprisingly,
open-loop controls and feedback controls will lead to different solutions.
In the case of partial information, we have unfortunately no choice.
There is no feedback, since the state is not accessible. With the
formulation above, the observation filtration $\mathcal{Z}^{t}$ is
externally defined , and the control is open-loop, since it is externally
defined as a process adapted to $\mathcal{Z}^{t}.$ It is important
to have this discussion in mind, when we formulate the problem with
mean-field theory.  
\end{rem}

\subsection{CONTROL OF ZAKAI EQUATION}

Note first that the functional (\ref{eq:2-6}) can be written as 

\begin{equation}
J(v(.))=E[\int_{0}^{T}\eta(t)f(x(t),v(t))dt+\eta(T)f_{T}(x(T))]\label{eq:2-7}
\end{equation}
This is obtained by using the Radon-Nikodym derivative (\ref{eq:2-4})
and the martingale property of $\eta(t).$ We next recall the classical
nonlinear filtering theory result. Let $\Psi(x)$ be any bounded continuous
function. We want to express the conditional expectation $E^{v(.)}[\Psi(x(t))|\mathcal{Z}^{t}]$of
the random variable $\Psi(x(t))$ with respect to the $\sigma-$algebra
$\mathcal{Z}^{t},$ on the probability space $\Omega,\mathcal{A},P^{v(.)}.$
We have the basic result of non linear filtering theory 

\begin{equation}
E^{v(.)}[\Psi(x(t))|\mathcal{Z}^{t}]=\frac{E[\eta(t)\Psi(x(t))|\mathcal{Z}^{t}]}{E[\eta(t)|\mathcal{Z}^{t}]}=\dfrac{\int_{R^{n}}\Psi(x)q(x,t)dx}{\int_{R^{n}}q(x,t)dx}\label{eq:2-8}
\end{equation}
where $q(x,t)$ is called the un-normalized conditional probaility
density of the random variable $x(t)$ with respect to the $\sigma-$algebra
$\mathcal{Z}^{t}$. The conditional probability itself is given by
$\dfrac{q(x,t)}{\int_{R^{n}}q(\xi,t)d\xi}.$ The function $q(x,t)$
is a random field adapted to the filtration $\mathcal{Z}^{t}.$It
is the solution of a stochastic P.D.E. 

\begin{equation}
dq+A^{*}q(x,t)dt+\text{div }(g(x,v(t))q(x,t))\,dt-q(x,t)\,h(x).dz(t)=0\label{eq:2-9}
\end{equation}

\[
q(x,0)=q_{0}(x)
\]
 in which $A^{*}$ is the second order differential operator 

\[
A^{*}\varphi(x)=-\sum_{i,j=1}^{n}\dfrac{\partial^{2}}{\partial x_{i}\partial x_{j}}(a_{ij}\varphi(x))
\]
which is the dual of 

\[
A\varphi(x)=-\sum_{i,j=1}^{n}a_{ij}(x)\dfrac{\partial^{2}\varphi}{\partial x_{i}\partial x_{j}}
\]

with $a(x)=\dfrac{1}{2}\sigma\sigma^{*}(x).$The initial condotion
$q_{0}(x)$ is the probability density of $\xi.$ We suppose that
$\xi$ has a probability density. The random field $q(x,t)$ depends
on $v(.)$ and is thus denoted $q^{v(.)}(x,t)$. From (\ref{eq:2-8})
and (\ref{eq:2-7}) we can write the pay-off $J(v(.))$ as 

\begin{equation}
J(v(.))=E[\int_{0}^{T}\int_{R^{n}}q^{v(.)}(x,t)f(x,v(t))dxdt+\int_{R^{n}}q^{v(.)}(x,T)f_{T}(x)dx]\label{eq:2-10}
\end{equation}
The minimization of $J(v(.))$ is a stochastic control problem for
a dynamic system whose evolution is governed by the stochastic P.D.E.
(\ref{eq:2-9}). 
\begin{rem}
\label{rem2-2}We can elaborate more on the difference between feedback
controls and open-loop controls , as addressed in Remark \ref{rem2-1},
by considering equation (\ref{eq:2-9}) describing the evolution of
the state $q(x,t).$ In this equation $v(t)$ is a stochastic process
adapted to the filtration $\mathcal{Z}^{t},$ so it is fixed with
respect to the space variable $x.$ 
\end{rem}

\section{\label{sec:MEAN-FIELD-APPROACH}MEAN FIELD APPROACH }

\subsection{PRELIMINARIES}

We define the value function 

\begin{equation}
\Phi(q_{0},0)=\inf_{v(.)}J(v(.))\label{eq:3-1}
\end{equation}
and following the main concept of Dynamic Programming, we embed this
value function into a family parametrized by initial conditions $q,t$
, where $q$ denotes an unnormalized probability density on $R^{n}.$We
also make precise the choice of the functional space in which the
function $q(x)$ lies. To fix ideas , we take $q\in L^{2}(R^{n})\cap L^{1}(R^{n})$
and $q(x)\geq0.$ We shall assume that 

\begin{equation}
\int_{R^{n}}|x|^{2}q(x)dx<+\infty\label{eq:3-2}
\end{equation}
Considering functionals on $L^{2}(R^{n})$ , $\Psi(q),$ we say that
it is Gateaux differentiable , with Gateaux derivative $\dfrac{\partial\Psi}{\partial q}(q)(x)$
if the function $t\rightarrow\Psi(q+t\tilde{q})$ is differentiable
with the formula 

\begin{equation}
\dfrac{d}{dt}\Psi(q+t\tilde{q})=\int_{R^{n}}\dfrac{\partial\Psi}{\partial q}(q+t\tilde{q})(x)\tilde{q}(x)dx,\:\forall\tilde{q}(.)\in L^{2}(R^{n})\cap L^{1}(R^{n}),\tilde{q}(x)\geq0\label{eq:3-3}
\end{equation}
We shall assume that with $q\times x\rightarrow\dfrac{\partial\Psi}{\partial q}(q)(x)$
is continuous , satisfying 

\begin{equation}
|\dfrac{\partial\Psi}{\partial q}(q)(x)|\leq c(q)(1+|x|^{2})\label{eq:2-104}
\end{equation}
such that $c(q)$ is continuous and bounded on bounded subsets of
$L^{2}(R^{n}),$ We also need the concept of second order Gateaux
derivative. The second order Gateaux derivative is a functional $\dfrac{\partial^{2}\Psi}{\partial q^{2}}(q)(\xi,\eta)$
such that the function $t\rightarrow$$\Psi(q+t\tilde{q})$ is twice
differentiable in $t$ and 

\begin{equation}
\dfrac{d^{2}}{dt^{2}}\Psi(q+t\tilde{q})=\int_{R^{n}}\dfrac{\partial^{2}\Psi}{\partial q^{2}}(q+t\tilde{q})(\xi,\eta)\tilde{q}(\xi)\tilde{q}(\eta)d\xi d\eta\label{eq:2-112}
\end{equation}
Moreover, the function $q,\xi,\eta\rightarrow$ $\dfrac{\partial^{2}\Psi}{\partial q^{2}}(q)(\xi,\eta)$
is continuous satisfying 

\begin{equation}
|\dfrac{\partial^{2}\Psi}{\partial q^{2}}(q)(\xi,\eta)|\leq c(q)(1+|\xi|^{2}+|\eta|^{2})\label{eq:2-113}
\end{equation}
with $c(q)$ continuous bounded on bounded subsets of $L^{2}(R^{n}),$
From formula (\ref{eq:2-112}), it is clear that we can choose $\dfrac{\partial^{2}\Psi}{\partial q^{2}}(q)(\xi,\eta)$
to be symmetric in $\xi,\eta.$ Set 

\[
f(t)=\Psi(q+t\tilde{q})
\]
 Then , combining the above assumptions, we can assert that $f(t)$
is $C^{2}.$Therefore , we have the identity 

\[
f(1)=f(0)+f'(0)+\int_{0}^{1}\int_{0}^{1}tf"(st)dsdt
\]
 which leads to the formula 

\begin{equation}
\Psi(q+\tilde{q})=\Psi(q)+\int_{R^{n}}\dfrac{\partial\Psi}{\partial q}(q)(x)\tilde{q}(x)dx+\int_{0}^{1}\int_{0}^{1}t\int_{R^{n}}\dfrac{\partial^{2}\Psi}{\partial q^{2}}(q+st\tilde{q})(\xi,\eta)\tilde{q}(\xi)\tilde{q}(\eta)d\xi d\eta\label{eq:2-114}
\end{equation}

\subsection{BELLMAN EQUATION }

We consider the control problem with initial conditions $q,t$ 

\begin{equation}
dq+A^{*}q(x,s)ds+\text{div }(g(x,v(s))q(x,s))\,ds-q(x,s)\,h(x).dz(s)=0,\;s>t\label{eq:3-4}
\end{equation}

\[
q(x,t)=q(x)
\]

\begin{equation}
J_{q,t}(v(.))=E[\int_{t}^{T}\int_{R^{n}}q^{v(.)}(x,s)f(x,v(s))dxds+\int_{R^{n}}q^{v(.)}(x,T)f_{T}(x)dx]\label{eq:3-5}
\end{equation}
and define the value function 

\begin{equation}
\Phi(q,t)=\inf_{v(.)}J_{q,t}(v(.))\label{eq:3-6}
\end{equation}
Assuming that the value function has derivatives 

\[
\dfrac{\partial\Phi}{\partial t}(q,t),\:\dfrac{\partial\Psi}{\partial q}(q)(x),\:\dfrac{\partial^{2}\Psi}{\partial q^{2}}(q)(\xi,\eta)
\]
 then, by standard arguemnts, we can check formally that $\Phi(q,t)$
is solution of the Bellman equation 

\begin{equation}
\dfrac{\partial\Phi}{\partial t}-\int_{R^{n}}A\,\dfrac{\partial\Phi}{\partial q}(q,t)(x)q(x)dx+\label{eq:3-7}
\end{equation}

\[
+\dfrac{1}{2}\int_{R^{n}}\int_{R^{n}}\dfrac{\partial^{2}\Phi}{\partial q^{2}}(q,t)(\xi,\eta)q(\xi)q(\eta)h(\xi).h(\eta)d\xi d\eta+
\]
\[
+\inf_{v}\int_{R^{n}}q(x)\left(f(x,v)+D_{x}\dfrac{\partial\Phi}{\partial q}(q,t)(x).g(x,v)\right)dx=0
\]
\[
\Phi(q,T)=\int_{R^{n}}f_{T}(x)q(x)dx
\]
 The optimal open-loop control is obtained by achieving the infimum
in (\ref{eq:3-7}) . We derive a functional $\hat{v}(q,t),$ which
is a feedback in $q$ but not in $x.$We can then feed the Zakai equation
(\ref{eq:3-4}) with this feedback to get the optimal state equation

\begin{equation}
dq+A^{*}q(x,s)ds+\text{div }(g(x,\hat{v}(q,s))q(x,s))\,ds-q(x,s)\,h(x).dz(s)=0,\:s>t\label{eq:3-8}
\end{equation}
\[
q(x,t)=q(x)
\]
Once we solve this stochastic P.D.E. we obtain the optimal state $\hat{q}(s):=\hat{q}(x,s).$
We then define the control $\hat{v}(s)=\hat{v}(\hat{q}(s),s)$, which
is indeed adapted to the filtration $\mathcal{Z}_{t}^{s}=\sigma(z(\tau)-z(t),\,t\leq\tau\leq s).$
This is the optimal open-loop control. 

\subsection{THE MASTER EQUATION}

The functional $\hat{v}(q,t)$ defined above depends on the function
$\dfrac{\partial\Phi}{\partial q}(q,t)(x)$ denoted $U(x,q,t).$ So
, it is convenient to denote by $\hat{v}(q,U)$ the vector $v$ which
achieves the minimum of 

\begin{equation}
\inf_{v}\int_{R^{n}}q(x)\left(f(x,v)+D_{x}U(x,q).g(x,v)\right)dx\label{eq:3-9}
\end{equation}
in which we omit to write explicity the arguement $t.$ Bellman equation
(\ref{eq:3-7}) can be written as 

\begin{equation}
\dfrac{\partial\Phi}{\partial t}-\int_{R^{n}}A_{x}\,U(x,q)\,q(x)dx+\label{eq:3-10}
\end{equation}

\[
+\dfrac{1}{2}\int_{R^{n}}\int_{R^{n}}\dfrac{\partial^{2}\Phi}{\partial q^{2}}(q,t)(\xi,\eta)q(\xi)q(\eta)h(\xi).h(\eta)d\xi d\eta+
\]
\[
+\int_{R^{n}}q(x)\left(f(x,\hat{v}(q,U))+D_{x}U(x,q).g(x,\hat{v}(q,U))\right)dx=0
\]
\[
\Phi(q,T)=\int_{R^{n}}f_{T}(x)q(x)dx
\]
It is also convenient to set 

\begin{equation}
V(q,t)(x,y)=\dfrac{\partial^{2}\Phi}{\partial q^{2}}(q,t)(x,y)\label{eq:3-11}
\end{equation}
Therefore , Bellman equation reads 

\begin{equation}
\dfrac{\partial\Phi}{\partial t}-\int_{R^{n}}A_{x}\,U(x,q)\,q(x)dx+\label{eq:3-12}
\end{equation}

\[
+\dfrac{1}{2}\int_{R^{n}}\int_{R^{n}}V(q,t)(\xi,\eta)q(\xi)q(\eta)h(\xi).h(\eta)d\xi d\eta+
\]
\[
+\int_{R^{n}}q(x)\left(f(x,\hat{v}(q,U))+D_{x}U(x,q).g(x,\hat{v}(q,U))\right)dx=0
\]
\[
\Phi(q,T)=\int_{R^{n}}f_{T}(x)q(x)dx
\]
The Master equation is an equation for $U(x,q,t).$ It is obtained
by differentiating (\ref{eq:3-12}) with respect to $q.$ We obtain
, formally 

\begin{equation}
\dfrac{\partial U}{\partial t}-A_{x}U-\int_{R^{n}}A_{\xi}V(q,t)(x,\xi)q(\xi)d\xi+\label{eq:3-13}
\end{equation}

\[
+h(x).\int_{R^{n}}V(q,t)(x,\xi)h(\xi)q(\xi)d\xi+\dfrac{1}{2}\int_{R^{n}}\int_{R^{n}}\dfrac{\partial V}{\partial q}(q,t)(\xi,\eta)(x)h(\xi).h(\eta)q(\xi)q(\eta)d\xi d\eta+
\]

\[
+f(x,\hat{v}(q,U))+D_{x}U.g(x,\hat{v}(q,U))+\int_{R^{n}}D_{\xi}V(q,t)(\xi,x).g(\xi,\hat{v}(q,U))q(\xi)d\xi=0
\]
\[
U(x,q,T)=f_{T}(x)
\]
Note that 

\begin{equation}
\dfrac{\partial V}{\partial q}(q,t)(\xi,\eta)(x)=\dfrac{\partial^{3}\Phi}{\partial q^{3}}(q,t)(x,\xi,\eta)\label{eq:3-14}
\end{equation}
which is symmetric in the arguments $(x,\xi,\eta)$

\subsection{SYSTEM OF HJB-FP EQUATIONS}

In Mean field theory approach, the Master equation is the key equation.
However, it is an infinite-dimensional nonlinear P.D.E. Direct approaches
are very limited. The most convenient approach is to use ideas similar
to the classical method of characteristics. This amounts to solving
a system of forward-backward finite dimensional stochastic P.D.E.
Since it is forward-backward the initial conditions matter. We shall
consider that the initial time is $0,$ for convenience. It should
be any time $t\in[0,T].$ This system is called Hamilton-Jacobi-Bellman
for the backward equation and Fokker-Planck for the forward one. The
Fokker-Planck equation is the Zakai equation in which we insert the
optimal feedabck $\hat{v}(q,U).$ So we get 

\begin{equation}
dq+A^{*}q(x,t)dt+\text{div }(g(x,\hat{v}(q,U))q(x,t))\,dt-q(x,t)\,h(x).dz(t)=0\label{eq:3-15}
\end{equation}

\[
q(x,0)=q(x)
\]
The functional $U(x,q,t)$ used in (\ref{eq:3-15}) is the functional
solution of the master equation (\ref{eq:3-13}). We call simply $q(t)$
the solution of (\ref{eq:3-15}). We then set 

\begin{equation}
u(x,t)=U(x,q(t),t)\label{eq:3-16}
\end{equation}
We use the notation $\hat{v}(q_{t},u_{t})$ to represent the functional
$\hat{v}(q,U)$ in which the arguments $q,U$ are replaced by $q(.,t)$
and $U(.,q(.,t),t)=u(.,t).$ We note $q_{t}=q(.,t)$$,u_{t}=u(.,t)$
to simplify. The functional $\hat{v}(q_{t},u_{t})$ achieves the infimum
of 

\begin{equation}
\hat{v}(q_{t},u_{t})=\text{Arg}\,\min_{v}\,\int_{R^{n}}q(x,t)\left(f(x,v)+D_{x}u(x,t).g(x,v)\right)dx\label{eq:3-17}
\end{equation}
The next step is to obtain the equation for $u(x,t).$ It is a long
and tedious calculation, obtained in taking the Ito differential of
the random field defined by (\ref{eq:3-16}) . We give the result
as follows 

\begin{equation}
-du+(Au\,-Du.g(x,\hat{v}(q_{t},u_{t})))dt=f(x,\hat{v}(q_{t},u_{t}))-K(x,t).(dz(t)-h(x)dt)\label{eq:3-18}
\end{equation}

\[
u(x,T)=f_{T}(x)
\]
 where $K(x,t)$ is defined by the formula 

\begin{equation}
K(x,t)=\int_{R^{n}}V(q_{t},t)(x,\xi)h(\xi)q(\xi,t)d\xi\label{eq:3-19}
\end{equation}
In fact , we do not need to compute $K(x,t)$ by formula (\ref{eq:3-19})
, which would require the knowledge of $V(q_{t},t)(x,\xi),$ thus
solving the master equation. From the theory of backward stochastic
P.D.E. the random field $K(x,t)$ is required by the condition of
adaptativity of $u(x,t).$ So the solution of (\ref{eq:3-18}) is
not just $u(x,t)$ but the pair $(u(x,t),K(x,t)),$and we can expect
uniqueness. The equation (\ref{eq:3-18}) is the HJB equation . It
must be coupled with the FP equation (\ref{eq:3-15}) written as follows 

\begin{equation}
dq+A^{*}q(x,t)dt+\text{div }(g(x,\hat{v}(q_{t},u_{t}))q(x,t))\,dt-q(x,t)\,h(x).dz(t)=0\label{eq:3-20}
\end{equation}

\[
q(x,0)=q(x)
\]
recalling also (\ref{eq:3-17}). So the pair (\ref{eq:3-18}), (\ref{eq:3-20})
is the pair of HJB-FP equations. Since $q(x,0)=q(x)$ we can assert 

\begin{equation}
u(x,0)=U(x,q,0)\label{eq:3-21}
\end{equation}
Therefore we can compute $U(x,q,0)$ by solving the system of HJB-FP
equations and using formula (\ref{eq:3-21}). Of course $u(x,t)\not=U(x,q,t).$
To compute $U(x,q,t)$ we have to write the system (\ref{eq:3-18}),
(\ref{eq:3-20}) on the interval $(t,T)$ instead of $(0,T).$ In
that sense, the system of HJB-FP equations (\ref{eq:3-18}), (\ref{eq:3-20})
is a method of characteristics to solve the master equation (\ref{eq:3-13}).
Besides the optimal optimal feedback $\hat{v}(q,U(.,q,t),t)$ can
be derived from the system of HJB-FP equations . Indeed, 

\[
\hat{v}(q,U(.,q,0),0)=\hat{v}(q,u(.,0))
\]
and setting the initial condition of the system of HJB-FP equations
at $t$ instead of $0$ yields $\hat{v}(q,U(.,q,t),t).$ To compute
the value function, we have to rely on Bellman equation (\ref{eq:3-12}).
Let us compute $\dfrac{\partial\Phi}{\partial t}(q,0),$ by using
(\ref{eq:3-12}). The only term which is not known is $\int_{R^{n}}\int_{R^{n}}V(q,0)(\xi,\eta)q(\xi)q(\eta)h(\xi).h(\eta)d\xi d\eta.$However
from (\ref{eq:3-19}) we can write 

\begin{equation}
\int_{R^{n}}\int_{R^{n}}V(q,0)(\xi,\eta)q(\xi)q(\eta)h(\xi).h(\eta)d\xi d\eta=\int_{R^{n}}h(\xi).K(\xi,0)q(\xi)d\xi\label{eq:3-22}
\end{equation}
Collecting results we can write the formula 

\begin{equation}
\dfrac{\partial\Phi}{\partial t}(q,0)=\int_{R^{n}}A_{x}\,u(x,0)\,q(x)dx-\dfrac{1}{2}\int_{R^{n}}h(x).K(x,0)q(x)dx\label{eq:3-23}
\end{equation}
\[
-\int_{R^{n}}q(x)\left(f(x,\hat{v}(q,u(.,0)))+D_{x}u(x,0).g(x,\hat{v}(q,u(.,0)))\right)dx
\]
In a similar way we can define $\dfrac{\partial\Phi}{\partial t}(q,t)$
for any $t$ and any $q.$ Since we know $\Phi(q,T)$ we obtain $\Phi(q,t)$
for any $t.$ So solving the system of HJB-FP equations provides all
the information on the value function and on the optimal feedback. 

\section{WEAK FORMULATION OF ZAKAI EQUATION}

\subsection{WEAK FORMULATION AND LINEAR DYNAMICS}

In this section, we consider Zakai equation as follows 

\begin{equation}
dq+A^{*}q(x,t)dt+\text{div }(g(x,v(t))q(x,t))\,dt-q(x,t)\,h(x).dz(t)=0\label{eq:4-1}
\end{equation}

\[
q(x,0)=q(x)
\]
in which $v(t)$ is a fixed process adapted to the filtration $\mathcal{Z}^{t}=\sigma(z(s),s\leq t).$
If $\psi(x,t)$ is a deterministic function of $x,t$ which is $C^{2}$
in $x$ and $C^{1}$ in $x,$ we deduce immediately from (\ref{eq:4-1}),
by simple integration by parts that 

\[
\int_{R^{n}}dq\psi(x,t)=\int_{R^{n}}q(x,t)[-A\psi(x,t)+g(x,v(t)).D\psi(x,t)]dt+\int_{R^{n}}q(x,t)\psi(x,t)h(x).dz(t)
\]
 and thus 

\begin{equation}
\int_{R^{n}}q(x,t)\psi(x,t)dx=\int_{R^{n}}q(x)\psi(x,0)dx+\label{eq:4-2}
\end{equation}

\[
\int_{0}^{t}\int_{R^{n}}q(x,s)(\dfrac{\partial\psi}{\partial s}-A\psi(x,s)+g(x,v(s)).D\psi(x,s))dxds+\int_{0}^{t}\int_{R^{n}}q(x,s)\psi(x,s)h(x).dz(s)
\]
 which is the weak formulation of Zakai equation . Note that the formulation
(\ref{eq:4-1}) (strong form) and the weak form (\ref{eq:4-2}) are
not equivalent. We may have a weak solution and not a strong solution. 

\subsection{LINEAR SYSTEM AND LINEAR OBSERVATION}

We want to solve Zakai equation in the following case 

\begin{equation}
g(x,v)=Fx+Gv,\:\sigma(x)=\sigma\label{eq:4-3}
\end{equation}
\[
h(x)=Hx
\]
 In general , this case is associated to an initial probability $q(x),$
which is gaussian. In our approach , we cannot take a special $q(x).$It
must remain general, because it is an arument of the value function
and of the solution of the master equation. When we solve the system
of HJB-FP equations, we can take $q(x)$ gaussian, but then we cannot
use this method to obtain the solution of the master equation or of
Bellman equation. For a given control $v(t)$ which is a process adapted
to $\mathcal{Z}^{t},$ Zakai equation reads 

\begin{equation}
dq-\text{tr}aD_{x}^{2}q(x,t)\:dt+\text{div }(Fx+Gv(t))q(x,t))\,dt-q(x,t)\,Hx.dz(t)=0\label{eq:4-4}
\end{equation}

\[
q(x,0)=q(x)
\]
where $a=\dfrac{1}{2}\sigma\sigma^{*}.$A. Makowsky \cite{ARM} has
shown that this equation has an explicit solution, that we describe
now, in a weak form. We first need some notation. We introduce the
matrix $\Sigma(t)$ solution of the Riccati equation 

\begin{equation}
\dfrac{d\Sigma}{dt}+\Sigma(t)H^{*}H\Sigma(t)-F\Sigma(t)-\Sigma(t)F^{*}=2a\label{eq:4-5}
\end{equation}

\[
\Sigma(0)=0
\]
 We then define the matrix $\Phi(t)$ solution of the differential
eqaution 

\begin{equation}
\dfrac{d\Phi}{dt}=(F-\Sigma(t)H^{*}H)\Phi(t)\label{eq:4-6}
\end{equation}

\[
\Phi(0)=I
\]
and 

\begin{equation}
S(t)=\int_{0}^{t}\Phi^{*}(s)H^{*}H\Phi(s)ds\label{eq:4-7}
\end{equation}
We then introuduce stochastic processes $\beta(t)$ and $\rho(t)$
adapted to the filtration $\mathcal{Z}^{t},$ defined by the equations 

\begin{equation}
d\beta(t)=(F\beta(t)+Gv(t))dt+\Sigma(t)H^{*}(dz-H\beta(t)dt)\label{eq:4-8}
\end{equation}

\[
\beta(0)=0
\]
 
\begin{equation}
d\rho(t)=\Phi^{*}(t)H^{*}(dz(t)-H\beta(t)dt)\label{eq:4-9}
\end{equation}
\[
\rho(0)=0
\]
 The process $\beta(t)$ is the Kalman filter for the linear system
(\ref{eq:4-3}) with a deterministic initial condition , equal to
$0.$ If we set 

\begin{equation}
m(x,t)=\Phi(t)x+\beta(t)\label{eq:4-10}
\end{equation}
we obtain the Kalman filter for the same linear dynamic system, with
intial condition $x.$ It satisfies the equation 

\begin{equation}
d_{t}m(x,t)=(Fm(x,t)+Gv(t))dt+\Sigma(t)H^{*}(dz-Hm(x,t)dt)\label{eq:4-11}
\end{equation}

\[
m(x,0)=x
\]
 Finally we introduce the martingale $\theta(x,t)$ defined by 

\begin{equation}
d_{t}\theta(x,t)=\theta(x,t)\,Hm(x,t).dz(t)\label{eq:4-12}
\end{equation}
\[
\theta(x,0)=1
\]
whose solution is the exponential 

\begin{equation}
\theta(x,t)=\exp\left(\int_{0}^{t}Hm(x,s).dz(s)-\tfrac{1}{2}\int_{0}^{t}|Hm(x,s).|^{2}ds\right)\label{eq:4-13}
\end{equation}

\subsection{FORMULAS }

We can state the following result , due to A. Makowsky \cite{ARM},
whose proof can be found in \cite{ABE} 
\begin{prop}
\label{prop4-1} For any test function $\psi(x,t)$ , we have 

\begin{equation}
\int_{R^{n}}q(x,t)\psi(x,t)dx=\int_{R^{n}}\theta(x,t)\left(\int_{R^{n}}\psi(m(x,t)+\Sigma(t)^{\frac{1}{2}}\xi,t)\dfrac{\exp-\dfrac{|\xi|^{2}}{2}}{(2\pi)^{\frac{n}{2}}}d\xi\right)q(x)dx\label{eq:4-14}
\end{equation}
\end{prop}

\begin{proof}
Equality (\ref{eq:4-14}) is true for $t=0.$ Let us set 

\[
\mathcal{L}\psi(x,t)=\dfrac{\partial\psi}{\partial s}-A\psi(x,t)+g(x,v(t)).D\psi(x,t)
\]
According to (\ref{eq:4-2}) it is thus sufficient to show that 

\[
d_{t}\theta(x,t)\left(\int_{R^{n}}\psi(m(x,t)+\Sigma(t)^{\frac{1}{2}}\xi,t)\exp-\dfrac{|\xi|^{2}}{2}d\xi\right)=\theta(x,t)\int_{R^{n}}\mathcal{L}\psi(m(x,t)+\Sigma(t)^{\frac{1}{2}}\xi,t)\exp-\dfrac{|\xi|^{2}}{2}d\xi dt+
\]

\[
+\theta(x,t)\int_{R^{n}}\psi(m(x,t)+\Sigma(t)^{\frac{1}{2}}\xi,t)H(m(x,t)+\Sigma(t)^{\frac{1}{2}}\xi)\exp-\dfrac{|\xi|^{2}}{2}d\xi.dz(t)
\]
 This is done through a tedious calculation, whose details can be
found in \cite{ABE} . $\blacksquare$ 
\end{proof}
We shall derive from (\ref{eq:4-14}) a more analytic formula. We
first set 

\begin{equation}
\nu(t)=\int_{R^{n}}q(x,t)dx=\int_{R^{n}}\theta(x,t)q(x)dx\label{eq:4-15}
\end{equation}
 Hence from (\ref{eq:4-12}) 

\[
d\nu(t)=\int_{R^{n}}\theta(x,t)Hm(x,t)q(x)dx.dz(t)
\]
 But from (\ref{eq:4-14}) we see that 

\begin{equation}
\int_{R^{n}}q(x,t)xdx=\int_{R^{n}}\theta(x,t)m(x,t)q(x)dx\label{eq:4-151}
\end{equation}
Therefore 

\begin{equation}
d\nu(t)=H\int_{R^{n}}q(x,t)xdx.dz(t)=\nu(t)H\hat{x}(t).dz(t)\label{eq:4-150}
\end{equation}
 where we have set 

\begin{equation}
\hat{x}(t)=\frac{\int_{R^{n}}q(x,t)xdx}{\int_{R^{n}}q(x,t)dx}\label{eq:4-16}
\end{equation}
 Referring to (\ref{eq:2-8}) we see that 

\[
\hat{x}(t)=E^{v(.)}[x(t)|\mathcal{Z}^{t}]
\]
 the conditional mean of the process $x(t)$ defined by , see (\ref{eq:2-1}) 

\begin{equation}
dx=(Fx(t)+Gv(t))dt+\sigma dw\label{eq:4-17}
\end{equation}
\[
x(0)=\xi
\]
with respect to the filtration $\mathcal{Z}^{t}$ on the probability
space $(\Omega,\mathcal{A},P^{v(.)})$ . It is thus the Kalman filter
in this probabilistic set up. We shall derive the form of its evolution
in the sequel. Now , from (\ref{eq:4-150}) we can assert that 

\begin{equation}
\nu(t)=\exp\left\{ \int_{0}^{t}H\hat{x}(s).dz(s)-\dfrac{1}{2}\int_{0}^{t}|H\hat{x}(s)|^{2}ds\right\} \int_{R^{n}}q(x)dx\label{eq:4-18}
\end{equation}
 recalling that , see (\ref{eq:4-15}), $\nu(0)=\int_{R^{n}}q(x)dx.$
Next, from (\ref{eq:4-13}) and (\ref{eq:4-10}) , we have 

\[
\theta(x,t)=\exp\left(\int_{0}^{t}H(\Phi(s)x+\beta(s)).dz(s)-\tfrac{1}{2}\int_{0}^{t}|H(\Phi(s)x+\beta(s)).|^{2}ds\right)
\]

\[
=\gamma(t)\exp-\frac{1}{2}(x^{*}S(t)x-2x^{*}\rho(t))
\]
with 

\[
\gamma(t)=\exp\left\{ \int_{0}^{t}H\beta(s).dz(s)-\dfrac{1}{2}\int_{0}^{t}|H\beta(s)|^{2}ds\right\} 
\]
and recalling the definition of $S(t)$ and $\rho(t)$, see ( \ref{eq:4-7})
and (\ref{eq:4-9}). From (\ref{eq:4-15}) we obtain 

\[
\nu(t)=\gamma(t)\int_{R^{n}}\exp-\frac{1}{2}(x^{*}S(t)x-2x^{*}\rho(t))\,q(x)dx
\]
 Combining results , we can assert that 

\begin{equation}
\theta(x,t)=\nu(t)\,\dfrac{\exp-\frac{1}{2}(x^{*}S(t)x-2x^{*}\rho(t))}{\int_{R^{n}}\exp-\frac{1}{2}(\xi^{*}S(t)\xi-2\xi^{*}\rho(t))\,q(\xi)d\xi}\label{eq:4-19}
\end{equation}
Next , using (\ref{eq:4-151}) and (\ref{eq:4-10}) we have 

\[
\int_{R^{n}}q(x,t)xdx=\int_{R^{n}}\theta(x,t)(\Phi(t)x+\beta(t))q(x)dx
\]

\[
=\Phi(t)\int_{R^{n}}\theta(x,t)x\,q(x)dx+\beta(t)\int_{R^{n}}\theta(x,t)\,q(x)dx
\]
 therefore , from (\ref{eq:4-16}) we obtain also 

\begin{equation}
\hat{x}(t)=\Phi(t)\dfrac{\int_{R^{n}}\theta(x,t)x\,q(x)dx}{\int_{R^{n}}\theta(x,t)\,q(x)dx}+\beta(t)\label{eq:4-20}
\end{equation}
Let us introduce the deterministic function of arguments $\rho\in R^{n}$
and $t$

\begin{equation}
b(\rho,t)=\dfrac{\int_{R^{n}}x\exp-\frac{1}{2}(x^{*}S(t)x-2x^{*}\rho)\,q(x)dx}{\int_{R^{n}}\exp-\frac{1}{2}(x^{*}S(t)x-2x^{*}\rho)\,q(x)dx}\label{eq:4-21}
\end{equation}
then (\ref{eq:4-20}) can be written 

\begin{equation}
\hat{x}(t)=\Phi(t)b(\rho(t),t)+\beta(t)\label{eq:4-22}
\end{equation}
We can finally state the main formula for the unnormalized conditional
probability $q(x,t)$ 
\begin{thm}
\label{theo4-1} The unnormalized conditional probability $q(x,t)$
is given by 

\begin{equation}
\int_{R^{n}}q(x,t)\psi(x,t)dx=\dfrac{\nu(t)}{\int_{R^{n}}\exp-\frac{1}{2}(x^{*}S(t)x-2x^{*}\rho(t))q(x)dx}\int_{R^{n}}\exp-\frac{1}{2}(x^{*}S(t)x-2x^{*}\rho(t))\times\label{eq:4-23}
\end{equation}

\[
\left[\int_{R^{n}}\psi(\hat{x}(t)+\Phi(t)(x-b(\rho(t),t))+\Sigma(t)^{\frac{1}{2}}\xi,t)\dfrac{\exp-\dfrac{|\xi|^{2}}{2}}{(2\pi)^{\frac{n}{2}}}d\xi\right]q(x)dx
\]
 
\end{thm}

\subsection{SUFFICIENT STATISTICS }

We see , from formula ( \ref{eq:4-23}) that the unnormalized conditional
probability $q(x,t)$ is completely characterized by two processes
$\hat{x}(t)$ and $\rho(t),$ which are stochastic processes adapted
to $\mathcal{Z}^{t}$ with values in $R^{n}.$ So it is important
to obtain their evolution. We need to introduce a new function $B(\rho,t)$
similar to $b(\rho,t)$ defined by the following formula

\begin{equation}
B(\rho,t)=\dfrac{\int_{R^{n}}xx^{*}\exp-\frac{1}{2}(x^{*}S(t)x-2x^{*}\rho)\,q(x)dx}{\int_{R^{n}}\exp-\frac{1}{2}(\xi^{*}S(t)\xi-2\xi^{*}\rho(t))\,q(\xi)d\xi}\label{eq:4-24}
\end{equation}
and we define 

\begin{equation}
\Gamma(\rho,t)=\Sigma(t)+\Phi(t)(B(\rho,t)-b(\rho,t)b^{*}(\rho,t))\Phi^{*}(t)\label{eq:4-25}
\end{equation}
We are going to show that 
\begin{prop}
\label{prop4-2} The pair $\hat{x}(t)$,$\rho(t)$ is solution of
the following system of S.D.E. 

\begin{equation}
d\hat{x}(t)=(F\hat{x}(t)+Gv(t))dt+\Gamma(\rho(t),t)H^{*}(dz(t)-H\hat{x}(t)dt)\label{eq:4-26}
\end{equation}

\[
\hat{x}(0)=\dfrac{\int_{R^{n}}xq(x)dx}{\int_{R^{n}}q(x)dx}
\]
 
\begin{equation}
d\rho(t)=\Phi^{*}(t)H^{*}\left(dz(t)-H(\hat{x}(t)-\Phi(t)b(\rho(t),t))\right)\label{eq:4-27}
\end{equation}

\[
\rho(0)=0
\]
 
\end{prop}

\begin{proof}
The pair $\rho(t),$$\beta(t)$ satisfies (\ref{eq:4-8}), (\ref{eq:4-9})
and $\hat{x}(t)$ satisfies (\ref{eq:4-22}) therefore

\begin{equation}
d\hat{x}(t)=d\beta(t)+\frac{d\Phi(t)}{dt}b(\rho(t),t)dt+\Phi(t)db(\rho(t),t)\label{eq:4-28}
\end{equation}
Next we use 

\begin{equation}
D_{\rho}b(\rho,t)=B(\rho,t)-b(\rho,t)b^{*}(\rho,t)\label{eq:4-29}
\end{equation}

\begin{equation}
\text{tr }(D_{\rho}^{2}b(\rho,t)L)=-B(\rho,t)(L+L^{*})b(\rho,t)-\text{tr }(B(\rho,t)L)\,b(\rho,t)+2b(\rho,t)\,b^{*}(\rho,t)Lb(\rho,t)+\label{eq:4-30}
\end{equation}

\[
+\dfrac{\int_{R^{n}}x\,x^{*}Lx\,\exp-\frac{1}{2}(x^{*}S(t)x-2x^{*}\rho)\,q(x)dx}{\int_{R^{n}}\exp-\frac{1}{2}(x^{*}S(t)x-2x^{*}\rho)\,q(x)dx}
\]
for any matrix $L.$ Therefore 

\[
\dfrac{1}{2}\text{tr }(D_{\rho}^{2}b(\rho,t)\Phi^{*}(t)H^{*}H\Phi(t))=-B(\rho,t)\Phi^{*}(t)H^{*}H\Phi(t)b(\rho,t)-\dfrac{1}{2}\text{tr}(B(\rho,t)\Phi^{*}(t)H^{*}H\Phi(t))\,b(\rho,t)+
\]
\[
+\text{tr }(b(\rho,t)\,b^{*}(\rho,t)\Phi^{*}(t)H^{*}H\Phi(t))b(\rho,t)+\dfrac{1}{2}\,\dfrac{\int_{R^{n}}x\,x^{*}\Phi^{*}(t)H^{*}H\Phi(t)x\,\exp-\frac{1}{2}(x^{*}S(t)x-2x^{*}\rho)\,q(x)dx}{\int_{R^{n}}\exp-\frac{1}{2}(x^{*}S(t)x-2x^{*}\rho)\,q(x)dx}
\]
Also 

\begin{equation}
\dfrac{\partial b(\rho,t)}{\partial t}=\dfrac{1}{2}\text{tr}(B(\rho,t)\Phi^{*}(t)H^{*}H\Phi(t))\,b(\rho,t)-\dfrac{1}{2}\,\dfrac{\int_{R^{n}}x\,x^{*}\Phi^{*}(t)H^{*}H\Phi(t)x\,\exp-\frac{1}{2}(x^{*}S(t)x-2x^{*}\rho)\,q(x)dx}{\int_{R^{n}}\exp-\frac{1}{2}(x^{*}S(t)x-2x^{*}\rho)\,q(x)dx}\label{eq:4-31}
\end{equation}
Hence 

\begin{equation}
\dfrac{\partial b(\rho,t)}{\partial t}+\dfrac{1}{2}\text{tr }(D_{\rho}^{2}b(\rho,t)\Phi^{*}(t)H^{*}H\Phi(t))=-B(\rho,t)\Phi^{*}(t)H^{*}H\Phi(t)b(\rho,t)+\label{eq:4-32}
\end{equation}
\[
+\text{tr }(b(\rho,t)\,b^{*}(\rho,t)\Phi^{*}(t)H^{*}H\Phi(t))b(\rho,t)
\]
 We can then compute $db(\rho(t),t)$ , making use of (\ref{eq:4-32}),
(\ref{eq:4-29}) and (\ref{eq:4-27}). We obtain 

\begin{equation}
db(\rho(t),t)=(B(\rho(t),t)-b(\rho(t),t)b^{*}(\rho(t),t))\Phi^{*}(t)H^{*}(dz(t)-H\hat{x}(t)dt)\label{eq:4-33}
\end{equation}
Using (\ref{eq:4-6}) , (\ref{eq:4-28}) and (\ref{eq:4-33}) we obtain
easily (\ref{eq:4-26}), recalling the definition of $\Gamma(\rho,t),$
see (\ref{eq:4-25}). The relation (\ref{eq:4-27}) follws immediately
from (\ref{eq:4-9}) and (\ref{eq:4-22}). The proof is complete.
$\blacksquare$

We also have the follwing interpretation of $\Gamma(\rho(t),t)$ as
the conditional variance of the process $x(t)$
\end{proof}
\begin{prop}
\label{prop4-3}We have the formula 

\begin{equation}
\Gamma(\rho(t),t)=\dfrac{\int_{R^{n}}xx^{*}q(x,t)dx}{\int_{R^{n}}q(x,t)dx}-\hat{x}(t)\hat{x}(t)^{*}\label{eq:4-34}
\end{equation}
 
\end{prop}

\begin{proof}
We use (\ref{eq:4-23}) to write 

\[
\dfrac{\int_{R^{n}}xx^{*}q(x,t)dx}{\int_{R^{n}}q(x,t)dx}=\dfrac{1}{\int_{R^{n}}\exp-\frac{1}{2}(x^{*}S(t)x-2x^{*}\rho(t))q(x)dx}\int_{R^{n}}\exp-\frac{1}{2}(x^{*}S(t)x-2x^{*}\rho(t))\times
\]
\[
\left[\int_{R^{n}}(\hat{x}(t)+\Phi(t)(x-b(\rho(t),t))+\Sigma(t)^{\frac{1}{2}}\xi)(\hat{x}(t)+\Phi(t)(x-b(\rho(t),t))+\Sigma(t)^{\frac{1}{2}}\xi)^{*}\dfrac{\exp-\dfrac{|\xi|^{2}}{2}}{(2\pi)^{\frac{n}{2}}}d\xi\right]q(x)dx
\]
\[
=\dfrac{1}{\int_{R^{n}}\exp-\frac{1}{2}(x^{*}S(t)x-2x^{*}\rho(t))q(x)dx}\int_{R^{n}}\exp-\frac{1}{2}(x^{*}S(t)x-2x^{*}\rho(t))\times
\]
\[
[\left((\hat{x}(t)+\Phi(t)(x-b(\rho(t),t))\right))\left((\hat{x}(t)+\Phi(t)(x-b(\rho(t),t))\right))^{*}+\Sigma(t)]=
\]
 
\[
=\hat{x}(t)\hat{x}(t)^{*}+\Phi(t)(B(\rho(t),t)-b(\rho(t),t))(B(\rho(t),t)-b(\rho(t),t))^{*}\Phi(t)^{*}+\Sigma(t)
\]
\[
=\hat{x}(t)\hat{x}(t)^{*}+\Gamma(\rho(t),t)
\]
 which is (\ref{eq:4-34}). $\blacksquare$ 
\end{proof}

\subsection{THE GAUSSIAN CASE }

We first begin by giving the characteristic function of the unnormalized
probability density ( Fourier transform ) denoted 

\begin{equation}
\hat{Q}(\lambda,t)=\int_{R^{n}}q(x,t)\exp i\lambda^{*}x\,dx=\label{eq:4-35}
\end{equation}

\[
\nu(t)\exp\left[-\dfrac{1}{2}\lambda^{*}\Sigma(t)\lambda+i\lambda^{*}(\hat{x}(t)-\Phi(t)b(\rho(t),t))\right]\times
\]
\[
\dfrac{\int_{R^{n}}\exp-\frac{1}{2}(x^{*}S(t)x-2x^{*}(\rho(t)+i\Phi^{*}(t)\lambda))\,q(x)dx}{\int_{R^{n}}\exp-\frac{1}{2}(x^{*}S(t)x-2x^{*}\rho(t))\,q(x)dx}
\]
 The gaussian case corresponds to an initial value of the system (\ref{eq:4-17})
which is gaussian 

\begin{equation}
q(x)=\dfrac{\exp-\dfrac{1}{2}(x-\bar{x}_{0})^{*}P_{0}^{-1}(x-\bar{x}_{0})}{(2\pi)^{\frac{n}{2}}|P_{0}|^{\frac{1}{2}}}\label{eq:4-36}
\end{equation}
where we have assumed the initial variance $P_{0}$ to be invertible
, to simplify calculations. Using (\ref{eq:4-21}) we obtain 

\begin{equation}
b(\rho,t)=(S(t)+P_{0}^{-1})^{-1}(\rho+P_{0}^{-1}\bar{x}_{0})\label{eq:4-37}
\end{equation}

\begin{equation}
B(\rho,t)=b(\rho,t)b(\rho,t)^{*}+(S(t)+P_{0}^{-1})^{-1}\label{eq:4-38}
\end{equation}
Therefore, from (\ref{eq:4-25}) we obtain 

\begin{equation}
\Gamma(\rho,t)=\Sigma(t)+\Phi(t)(S(t)+P_{0}^{-1})^{-1}\Phi(t)^{*}\label{eq:4-39}
\end{equation}
\[
=P(t)
\]
 which is independent of $\rho.$ An easy calculation shows that $P(t)$
is the solution of the Riccati equation 

\begin{equation}
\dfrac{dP}{dt}+PH^{*}HP-FP-PF^{*}=2a\label{eq:4-40}
\end{equation}

\[
P(0)=P_{0}
\]
 and $\hat{x}(t)$ is then the classical Kalman filter 

\begin{equation}
d\hat{x}(t)=(F\hat{x}(t)+Gv(t))dt+P(t)H^{*}(dz(t)-H\hat{x}(t)dt)\label{eq:4-41}
\end{equation}

\[
\hat{x}(0)=\bar{x}_{0}
\]
 To obtain $q(x,t)$ , we use the characteristic function (\ref{eq:4-35}).
An easy calculation yields 

\begin{equation}
\hat{Q}(\lambda,t)=\nu(t)\exp[i\lambda^{*}\hat{x}(t)-\dfrac{1}{2}\lambda^{*}P(t)\lambda]\label{eq:4-42}
\end{equation}
which is the characteristic function of a gaussian random variable
with mean $\hat{x}(t)$ and variance $P(t).$ Recall that it is a
conditional probability given $\mathcal{Z}^{t}.$ 

\section{LINEAR QUADRATIC CONTROL PROBLEM }

\subsection{SETTING OF THE PROBLEM }

We want to apply the theory developed in section \ref{sec:MEAN-FIELD-APPROACH}
to the linear dynamics and linear observation (\ref{eq:4-3}), with
a quadratic cost 

\begin{equation}
f(x,v)=x^{*}Mx+v^{*}Nv\label{eq:5-1}
\end{equation}

\[
f_{T}(x)=x^{*}M_{T}x
\]
 in which $M,$$M_{T}$ are $n\times n$ symmetric positive semi-definite
matrices and $N$ is a $m\times m$ symmetric positive definte matrix
. We want to solve the control problem (\ref{eq:2-9}), (\ref{eq:2-10})
in this case . We write it as follows 

\begin{equation}
dq-\text{tr}aD_{x}^{2}q(x,t)\:dt+\text{div }(Fx+Gv(t))q(x,t))\,dt-q(x,t)\,Hx.dz(t)=0\label{eq:5-2}
\end{equation}

\[
q(x,0)=q(x)
\]

\begin{equation}
J(v(.))=E[\int_{0}^{T}\int_{R^{n}}q^{v(.)}(x,t)(x^{*}Mx+v(t)^{*}Nv(t))dxdt+\int_{R^{n}}q^{v(.)}(x,T)x^{*}M_{T}xdx]\label{eq:5-3}
\end{equation}
In the sequel we will drop the index $v(.)$ in $q(x,t).$ 

\subsection{APPLICATION OF MEAN FIELD THEORY }

We begin by finding the function $\hat{v}(q,U)$ defined by (\ref{eq:3-9})
. We have to solve the minimization problem 

\begin{equation}
\inf_{v}[v^{*}Nv\int_{R^{n}}q(x)dx+Gv.\int_{R^{n}}D_{x}U(x,q)\,q(x)dx]\label{eq:5-4}
\end{equation}
which yields 

\begin{equation}
\hat{v}(q,U)=-\dfrac{1}{2}N^{-1}G^{*}\dfrac{\int_{R^{n}}D_{x}U(x,q)\,q(x)dx}{\int_{R^{n}}q(x)dx}\label{eq:5-5}
\end{equation}
and thus 

\begin{equation}
\inf_{v}[v^{*}Nv\int_{R^{n}}q(x)dx+Gv.\int_{R^{n}}D_{x}U(x,q)\,q(x)dx]=\label{eq:5-6}
\end{equation}

\[
-\dfrac{1}{4}\dfrac{\int_{R^{n}}D_{x}U(x,q)\,q(x)dx.GN^{-1}G^{*}\int_{R^{n}}D_{x}U(x,q)\,q(x)dx}{\int_{R^{n}}q(x)dx}
\]
 We consider the value function 

\begin{equation}
\Phi(q,t)=\inf_{v(.)}J(v(.))\label{eq:5-7}
\end{equation}
 and we write Bellman equation (\ref{eq:3-12}) 

\begin{equation}
\dfrac{\partial\Phi}{\partial t}+\text{tr }a\int_{R^{n}}D_{x}^{2}U(x,q,t)\,q(x)dx+\label{eq:5-8}
\end{equation}

\[
+\frac{1}{2}\int_{R^{n}}\int_{R^{n}}V(q,t)(\xi,\eta)H\xi.H\eta\,q(\xi)q(\eta)d\xi d\eta+\int_{R^{n}}x^{*}Mx\,q(x)dx+\int_{R^{n}}Fx.D_{x}U(x,q,t)q(x)dx
\]

\[
-\dfrac{1}{4}\dfrac{\int_{R^{n}}D_{x}U(x,q,t)\,q(x)dx.GN^{-1}G^{*}\int_{R^{n}}D_{x}U(x,q,t)\,q(x)dx}{\int_{R^{n}}q(x)dx}=0
\]
\[
\Phi(q,T)=\int_{R^{n}}x^{*}M_{T}x\,q(x)dx
\]
 in which we recall the notation

\[
U(x,q,t)=\dfrac{\partial\Phi(q,t)}{\partial q}(x),\;V(q,t)(x,y)=\dfrac{\partial^{2}\Phi(q,t)}{\partial q^{2}}(x,y)
\]
We can next write the Master equation (\ref{eq:3-13}), which is the
equation for $U(x,q,t).$ We get 

\begin{equation}
\dfrac{\partial U}{\partial t}+\text{tr }a\,D_{x}^{2}U(x,q,t)+\text{tr }a\int_{R^{n}}D_{\xi}^{2}V(q,t)(x,\xi)\,q(\xi)d\xi\label{eq:5-9}
\end{equation}

\[
+Hx.H\int_{R^{n}}\xi V(q,t)(x,\xi)\,q(\xi)d\xi+\frac{1}{2}\int_{R^{n}}\int_{R^{n}}\dfrac{\partial V(q,t)}{\partial q}(\xi,\eta)(x)H\xi.H\eta\,q(\xi)q(\eta)d\xi d\eta
\]

\[
+x^{*}Mx+Fx.D_{x}U(x,q,t)+\int_{R^{n}}F\xi.D_{\xi}V(q,t)(x,\xi)q(\xi)d\xi+
\]
\[
+\dfrac{1}{4}\dfrac{\int_{R^{n}}D_{x}U(x,q,t)\,q(x)dx.GN^{-1}G^{*}\int_{R^{n}}D_{x}U(x,q,t)\,q(x)dx}{(\int_{R^{n}}q(x)dx)^{2}}-\dfrac{1}{2}D_{x}U(x,q,t).GN^{-1}G^{*}\dfrac{\int_{R^{n}}D_{\xi}U(\xi,q,t)\,q(\xi)d\xi}{\int_{R^{n}}q(\xi)d\xi}
\]
\[
-\dfrac{1}{2}\int_{R^{n}}D_{\xi}V(q,t)(x,\xi)q(\xi)d\xi.GN^{-1}G^{*}\dfrac{\int_{R^{n}}D_{\xi}U(\xi,q,t)\,q(\xi)d\xi}{\int_{R^{n}}q(\xi)d\xi}=0
\]
\[
U(x,q,T)=x^{*}M_{T}x
\]

\subsection{SYSTEM OF HJB-FP EQUATIONS}

We now write the system of HJB-FP equations (\ref{eq:3-18}), (\ref{eq:3-20})
. We look for a pair $u(x,t),$$q(x,t)$ adapted random fields solution
of the coupled system 

\begin{equation}
-d_{t}u=\left(\text{tr }aD_{x}^{2}u+x^{*}Mx+Fx.D_{x}u(x,t)+\dfrac{1}{4}\dfrac{\int_{R^{n}}D_{\xi}u(\xi,t)\,q(\xi,t)d\xi.GN^{-1}G^{*}\int_{R^{n}}D_{\xi}u(\xi,t)\,q(\xi,t)d\xi}{(\int_{R^{n}}q(\xi,t)d\xi)^{2}}\right.\label{eq:5-10}
\end{equation}
\[
-\left.\frac{1}{2}D_{x}u(x,t).GN^{-1}G^{*}\frac{\int_{R^{n}}D_{\xi}u(\xi,t)\,q(\xi,t)d\xi}{\int_{R^{n}}q(\xi,t)d\xi}\right)dt-K(x,t).(dz(t)-Hxdt)
\]
\[
u(x,T)=x^{*}M_{T}x
\]
 
\begin{equation}
d_{t}q=\left(\text{tr }aD_{x}^{2}q-\text{div}[(Fx-\frac{1}{2}GN^{-1}G^{*}\frac{\int_{R^{n}}D_{\xi}u(\xi,t)\,q(\xi,t)d\xi}{\int_{R^{n}}q(\xi,t)d\xi})q(x,t)]\right)dt\label{eq:5-11}
\end{equation}

\[
+q(x,t)Hx.dz(t)
\]
 
\[
q(x,0)=q(x)
\]

The random field $K(x,t)$ can be expressed by 

\begin{equation}
K(x,t)=H\int_{R^{n}}\xi V(q_{t},t)(x,\xi)q(\xi,t)d\xi\label{eq:5-12}
\end{equation}

The key result is that we can solve this system of equations explicitly
and obtain the optimal control. We introduce the matrix $\pi(t)$
solution of the Riccati equation 

\begin{equation}
\dfrac{d\pi}{dt}+\pi F+F^{*}\pi-\pi GN^{-1}G^{*}\pi+M=0\label{eq:5-13}
\end{equation}

\[
\pi(T)=M_{T}
\]
 We next introduce the function $Z(x,\rho,t)$ solution of the deterministic
linear P.D.E. 

\begin{equation}
\dfrac{\partial Z}{\partial t}+D_{x}Z.(F-\Gamma(\rho,t)H^{*}H)x+D_{\rho}Z.\Phi^{*}(t)H^{*}H(\Phi(t)b(\rho,t)+x)+\label{eq:5-14}
\end{equation}

\[
+\dfrac{1}{2}\text{tr}D_{x}^{2}Z\,(2a+\Gamma(\rho,t)H^{*}H\Gamma(\rho,t))+\dfrac{1}{2}\text{tr}D_{\rho}^{2}Z\Phi^{*}(t)H^{*}H\Phi(t)-\text{tr}\,D_{x\rho}^{2}Z\Phi^{*}(t)H^{*}H\Gamma(\rho,t)+
\]

\[
+x^{*}\pi(t)GN^{-1}G^{*}\pi(t)x+2\text{tr}a\pi(t)=0
\]

\[
Z(x,\rho,T)=0
\]
 We next introduce the pair of adapted processes $\hat{x}(t),\rho(t)$
solution of the system of SDE 

\begin{equation}
d\hat{x}=(F-GN^{-1}G^{*}\pi(t))\hat{x}(t)dt+\Gamma(\rho(t),t)H^{*}(dz(t)-H\hat{x}(t)dt)\label{eq:5-15}
\end{equation}

\[
\hat{x}(0)=\bar{x}_{0}
\]

\begin{equation}
d\rho=\Phi^{*}(t)H^{*}\left(dz(t)-H(\hat{x}(t)-b(\rho(t),t)\right))\label{eq:5-16}
\end{equation}

\[
\rho(0)=0
\]
They are built on a convenient probablity space on which $z(t)$ is
a standard Wiener process with values in $R^{d}.$ We associate to
the pair $\hat{x}(t),\rho(t)$ the unnormalized conditional probability
$q(x,t)$ defined by Zakai equation 

\begin{equation}
d_{t}q=\left(\text{tr }aD_{x}^{2}q-\text{div}[(Fx-GN^{-1}G^{*}\pi(t)\hat{x}(t))q(x,t)]\right)dt+q(x,t)Hx.dz(t)\label{eq:5-150}
\end{equation}

\[
q(x,0)=q(x)
\]
 We next define the random field 

\begin{equation}
u(x,t)=x^{*}\pi(t)x+Z(x-\hat{x}(t),\rho(t),t)\label{eq:5-151}
\end{equation}
We state the main result of the paper 
\begin{thm}
\label{theo5-1}We have the property 

\begin{equation}
\int_{R^{n}}D_{x}Z(x-\hat{x}(t),\rho(t),t)q(x,t)dx=0,\:\text{a.s. },\forall t\label{eq:5-17}
\end{equation}
 and $u(x,t)$, $q(x,t)$ defined by (\ref{eq:5-151}), (\ref{eq:5-150})
are solution of (\ref{eq:5-10}), (\ref{eq:5-11}) . 

The optimal control is given by 

\begin{equation}
\hat{v}(t)=\hat{v}(q_{t},u_{t})=-N^{-1}G^{*}\pi(t)\hat{x}(t)\label{eq:5-180}
\end{equation}
\end{thm}

\begin{proof}
We first prove (\ref{eq:5-17}). We differentiate (\ref{eq:5-14})
in $x,$ to obtain 

\begin{equation}
\dfrac{\partial}{\partial t}D_{x}Z+D_{x}^{2}Z(F-\Gamma(\rho,t)H^{*}H)x+(F^{*}-H^{*}H\Gamma(\rho,t))D_{x}Z+\label{eq:5-19}
\end{equation}

\[
+D_{\rho x}^{2}Z\Phi^{*}(t)H^{*}H(\Phi(t)b(\rho,t)+x)+H^{*}H\Phi(t)D_{\rho}Z+
\]
\[
+\dfrac{1}{2}\text{tr}D_{x}^{2}\,D_{x}Z(2a+\Gamma(\rho,t)H^{*}H\Gamma(\rho,t))+\dfrac{1}{2}\text{tr}D_{\rho}^{2}\,D_{x}Z\Phi^{*}(t)H^{*}H\Phi(t)-\text{tr}\,D_{x\rho}^{2}\,D_{x}Z\Phi^{*}(t)H^{*}H\Gamma(\rho,t)+
\]

\[
+2\pi(t)GN^{-1}G^{*}\pi(t)x=0
\]
\[
Z(x,\rho,T)=0
\]
We next consider $q(x,t)$ defined by (\ref{eq:5-150}). A long calculation
then shows that 

\begin{equation}
d_{t}\int_{R^{n}}D_{x}Z(x-\hat{x}(t),\rho(t),t)q(x,t)dx=-(F^{*}-H^{*}H\Gamma(\rho(t),t)+H^{*}H\Phi(t))\int_{R^{n}}D_{x}Z(x-\hat{x}(t),\rho(t),t)q(x,t)dxdt+\label{eq:5-20}
\end{equation}

\[
+\left(\int_{R^{n}}[-D_{x}^{2}Z(x-\hat{x}(t),\rho(t),t)\Gamma(\rho(t),t)+D_{x\rho}^{2}Z(x-\hat{x}(t),\rho(t),t)\Phi^{*}(t)+D_{x}Z(x-\hat{x}(t),\rho(t),t)x^{*}]q(x,t)\right)H^{*}dz(t)
\]
 
\[
\int_{R^{n}}D_{x}Z(x-\hat{x}(T),\rho(T),T)q(x,T)dx=0
\]
 From this relation it follows that 

\[
\int_{R^{n}}D_{x}Z(x-\hat{x}(t),\rho(t),t)q(x,t)dx\,\exp\int_{0}^{t}(F^{*}-H^{*}H\Gamma(\rho(s),s)+H^{*}H\Phi(s))ds
\]
 is a $\mathcal{Z}^{t}$ martingale . Since it vanishes at $T,$ it
is $0$ at any $t,$ a.e. Hence (\ref{eq:5-17}) is obtained.Consider
$u(x,t)$ by formula (\ref{eq:5-151}) , therefore, using (\ref{eq:5-17})
we get 

\begin{equation}
\dfrac{\int_{R^{n}}D_{\xi}u(\xi,t)\,q(\xi,t)d\xi}{\int_{R^{n}}q(\xi,t)d\xi}=2\pi(t)\hat{x}(t)\label{eq:5-201}
\end{equation}
To check (\ref{eq:5-10}) we have to check 

\begin{equation}
-d_{t}u=[\text{tr }aD_{x}^{2}u+x^{*}Mx+Fx.D_{x}u+\hat{x}(t)^{*}\pi(t)GN^{-1}G^{*}\pi(t)\hat{x}(t)-\label{eq:5-21}
\end{equation}

\[
-D_{x}u.GN^{-1}G^{*}\pi(t)\hat{x}(t)]dt-K(x,t).(dz(t)-Hxdt)
\]
\[
u(x,T)=x^{*}M_{T}x
\]
 Note that the final condition is trivially satisfied. We can check
(\ref{eq:5-21}) by direct calculation. We obtain also the value of
$K(x,t)$

\begin{equation}
K(x,t)=H[-\Gamma(\rho(t),t)D_{x}Z(x-\hat{x}(t),\rho(t),t)+\Phi(t)D_{\rho}Z(x-\hat{x}(t),\rho(t),t)]\label{eq:5-22}
\end{equation}
So we have proved that $u(x,t),q(x,t)$ is solution of the system
of HJB-FP equations (\ref{eq:5-10}), (\ref{eq:5-11}). The result
(\ref{eq:5-180}) is an immediate consequence of (\ref{eq:5-201}).
The proof is complete. $\blacksquare$
\end{proof}

\subsection{COMPLEMENTS}

The result (\ref{eq:5-180}) is important. It shows that the optimal
control of the problem (\ref{eq:5-2}), (\ref{eq:5-3}) follows the
celabrated `` Separation Principle'' . We recall that in the deterministic
case , the optimal control, which is necessarily an open-loop control
can be obtained by a linear feedback on the state. Open loop and feedback
controls are equivalent. The separation principle claims that in the
partially observed case, the optimal open loop control ( adapted to
the observation process) can be obtained by the same feedback as in
the deterministic case, replacing the nonobservable state by its best
estimate, the Kalman filter. The fact that the separation principle
holds is well known when the initial state follows a gaussian distribution.
We have proven that it holds in general. What drives the separation
principle is the linearity of the dynamics and of the observation
and the fact that the cost is quadratic. The gaussian assumption does
not play any role. A significant simplification occurs in the gaussian
case, regarding the computation of the Kalman filter. In the gaussian
case, the Kalman filter solves a single equation. In general, the
Kalman filter is coupled to another statistics $\rho(t)$ and the
pair $\hat{x}(t)$, $\rho(t)$ must be obtained simultaneously. 

We proceed by obtaining the value function $\Phi(q,t).$ In fact ,
We obtain $\Phi(q,0).$ The same procedure must be repeated at any
time $t.$ First, we have 

\begin{equation}
\dfrac{\partial\Phi(q,0)}{\partial q}(x)=u(x,0)\label{eq:5-23}
\end{equation}
\[
=x^{*}\pi(0)x+Z(x-\dfrac{\int_{R^{n}}\xi q(\xi)d\xi}{\int_{R^{n}}q(\xi)d\xi},0,0)
\]

We next obtain $\dfrac{\partial\Phi}{\partial t}(q,0)$ by formula
(\ref{eq:3-23}). We obtain 

\begin{equation}
\dfrac{\partial\Phi}{\partial t}(q,0)=-\text{tr}a\int_{R^{n}}D_{x}^{2}u(x.0)q(x)dx-\dfrac{1}{2}\int_{R^{n}}Hx.K(x,0)q(x)dx-\label{eq:5-24}
\end{equation}

\[
-\int_{R^{n}}[x^{*}Mx+\hat{v}(0)^{*}N\hat{v}(0)+D_{x}u(x,0).(Fx+G\hat{v}(0))]q(x)dx
\]
\[
=-2\text{tr }a\pi(0)-\text{tr}a\,\left(\int_{R^{n}}D_{x}^{2}Z(x-\bar{x}_{0},0,0)q(x)dx\right)-\dfrac{1}{2}\int_{R^{n}}Hx.H(-\Gamma(0,0)D_{x}Z(x-\bar{x}_{0},0,0)+D_{\rho}Z(x-\bar{x}_{0},0,0))q(x)dx-
\]
\[
-\int_{R^{n}}Fx.D_{x}Z(x-\bar{x}_{0},0,0)q(x)dx+\left(\bar{x}_{0}^{*}\pi'(0)\bar{x}_{0}-\text{tr }\Gamma(0,0)(M+2\pi(0)F)\right)\int_{R^{n}}q(x)dx
\]
 where 

\[
\bar{x}_{0}=\frac{\int_{R^{n}}xq(x)dx}{\int_{R^{n}}q(x0dx}
\]
 We can finally obtain the value of $\Phi(q,0).$ Since we know the
optimal control for the problem (\ref{eq:5-2}), (\ref{eq:5-3}) we
have 

\[
\Phi(q,0)=J(\hat{v}(.))=
\]

\[
=E[\int_{0}^{T}\int_{R^{n}}q(x,t)(x^{*}Mx+\hat{x}(t)^{*}\pi(t)GN^{-1}G^{*}\pi(t)\hat{x}(t))dt+\int_{R^{n}}q(x,T)x^{*}M_{T}xdx]
\]
 where $q(x,t)$ is the solution of (\ref{eq:5-150}). Therefore also 

\begin{equation}
\Phi(q,0)=E\int_{0}^{T}\nu(t)[\hat{x}(t)^{*}(M+\pi(t)GN^{-1}G^{*}\pi(t))\hat{x}(t)+\text{tr}M\Gamma(\rho(t),t)]dt+\label{eq:5-26}
\end{equation}

\[
+E[\nu(T)(\hat{x}(T)^{*}M_{T}\hat{x}(T)+\text{tr}M\Gamma(\rho T),T)]
\]

The triple $\hat{x}(t),\rho(t),\nu(t)$ is solution of the system
of $S.D.E.$ (\ref{eq:5-15}),(\ref{eq:5-16}) and 

\begin{equation}
d\nu(t)=\nu(t)\,H\hat{x}(t).dz(t)\label{eq:5-25}
\end{equation}

\[
\nu(0)=\int_{R^{n}}q(x)dx
\]
 From the probabilistic formula (\ref{eq:5-26}) it is easy to derive
an analytic formula as follows 

\begin{equation}
\Phi(q,0)=(\bar{x}_{0}^{*}\pi(0)\bar{x}_{0}+\mu(0,0))\int_{R^{n}}q(x)dx\label{eq:5-27}
\end{equation}
 where $\mu(\rho,t)$ is the solution of the linear P.D.E. 

\begin{equation}
\dfrac{\partial\mu}{\partial t}+D_{\rho}\mu.\Phi^{*}(t)H^{*}H\Phi(t)b(\rho,t)+\label{eq:5-28}
\end{equation}

\[
+\dfrac{1}{2}\text{tr}D_{\rho}^{2}\mu\Phi^{*}(t)H^{*}H\Phi(t)+\text{tr}(M+\pi(t)\Gamma(\rho,t)H^{*}H)\Gamma(\rho,t)=0
\]
\[
\mu(\rho,T)=\text{tr}M_{T}\Gamma(\rho,T)
\]
 We can apply these results in the gaussian case. We first solve the
P.D.E s (\ref{eq:5-14}) and (\ref{eq:5-28}). We recall that $\Gamma(\rho,t)=P(t),$
then $Z(x,\rho,t)$ and $\mu(\rho,t)$ are independent of $\rho$
and as easily seen 

\begin{equation}
Z(x,\rho,t)=x^{*}\Lambda(t)x+\beta(t)\label{eq:5-29}
\end{equation}
 with 

\begin{equation}
\dfrac{d\Lambda}{dt}+\Lambda(t)(F-P(t)H^{*}H)+(F^{*}-H^{*}HP(t))\Lambda(t)+\pi(t)GN^{-1}G^{*}\pi(t)=0\label{eq:5-30}
\end{equation}

\[
\Lambda(T)=0
\]
 
\begin{equation}
\beta(t)=\int_{t}^{T}\text{tr}\Lambda(s)(2a+P(s)H^{*}HP(s))ds\label{eq:5-31}
\end{equation}
Similarly $\mu(\rho,t)=\mu(t)$ given by 

\begin{equation}
\mu(t)=\int_{t}^{T}\text{tr}\pi(s)P(s)H^{*}HP(s)ds+\text{tr }M_{T}P(T)\label{eq:5-32}
\end{equation}
Then 

\begin{equation}
\dfrac{\partial\Phi(q,0)}{\partial q}(x)=x^{*}\pi(0)x+(x-\bar{x}_{0})^{*}\Lambda(0)(x-\bar{x}_{0})+\beta(0)\label{eq:5-33}
\end{equation}

\begin{equation}
\dfrac{\partial\Phi}{\partial t}(q,0)=-2\text{tr }(a+P(0)F)(\pi(0)+\Lambda(0))\label{eq:5-34}
\end{equation}

\[
+\bar{x}_{0}^{*}\pi'(0)\bar{x}_{0}-\text{tr }P(0)M+\text{tr}P(0)H^{*}HP(0)\Lambda(0)
\]
 
\begin{equation}
\Phi(q,0)=\bar{x}_{0}^{*}\pi(0)\bar{x}_{0}+\mu(0)\label{eq:5-35}
\end{equation}

\section{COMPARAISON}

\subsection{OBJECTIVES}

We compare in this section our work with the approach of E. Bandini,
A. Cosso, M. Fuhrmann, H. Pham \cite{BCFP}. They consider a general
problem of stochastic control with partial information, to which our
problem can be reduced. Their set up leads to a conditional probability,
( hence normalized) solution of a linear stochastic PDE, which they
call DMZ equation ( for Duncan- Mortensen-Zakai equation). They formulate
a control problem for this infinite dimensional state equation, for
which they write a Bellman equation. The solution is a functional
on the Wasserstein space of probability measures, since indeed the
state is a probability. When we formulate our problem in their set
up, our Zakai equation cannot be their DMZ equation, since we have
not a probability , but an un-normalized probability. To make the
comparison easy we keep our model, but we follow the set up of \cite{BCFP}.
We explain the difference between the two equations ,and also the
difference between our Bellman equation and their Bellman equation.
Although our problem can appear as a particular case of \cite{BCFP},
it is at the price of complicating it, which turns out to be not suitable.
The discussion will explain the reasons. E. Bandini et al. provide
an example with linear equations, which does not cover ours. In the
set up of E. Bandini et al. our system remains nonlinear, which is
also a consequence of the complication of the approach . We remain
formal in our presentation, since we want to discuss the concepts
and compare the methods. 

\subsection{USE OF THE SET UP OF \cite{BCFP}}

In the set up of \cite{BCFP}, we consider the pair $x(t),\eta(t)$
solution of the system 

\begin{equation}
dx=g(x,v)dt+\sigma(x)dw\label{eq:6-1}
\end{equation}

\[
x(0)=x_{0}
\]

\begin{equation}
d\eta(t)=\eta(t)h(x(t)).dz(t)\label{eq:6-2}
\end{equation}

\[
\eta(0)=\eta_{0}
\]
 in which $w(.),z(.)$ are independent Wiener processes and $x_{0},\eta_{0}$
are random variables independendent of $w(.),z(.).$ We observe only
the process $z(.).$ The DMZ equation introduced by \cite{BCFP} is
the equation for the conditional probability of the pair $x(t),\eta(t)$
given the $\sigma-$ algebra $\mathcal{Z}^{t}$$=\sigma(z(s),s\leq t).$
In (\ref{eq:6-1}) the control $v(t)$ is simply adapted to $\mathcal{Z}^{t}$
. 

If $\varphi(\eta,x,t)$ is a deterministic function on $R^{n+1}\times R^{+}$
, we are interested in the process $\rho(\varphi)(t)=E[\varphi(\eta(t),x(t),t)|\mathcal{Z}^{t}].$
It is the solution of the DMZ equation. We use the notation 

\begin{equation}
A_{x}\varphi(\eta,x,t)=-\text{tr}(D_{x}^{2}\varphi(\eta,x,t)a(x))\label{eq:6-3}
\end{equation}
with $a(x)=\frac{1}{2}\sigma(x)\sigma^{*}(x).$ We note 

\begin{equation}
A_{x}^{*}\varphi(\eta,x,t)=-\sum_{ij}\dfrac{\partial^{2}}{\partial x_{i}\partial x_{j}}(a_{ij}(x)\varphi(\eta,x,t))\label{eq:6-4}
\end{equation}
We next define the operators 

\begin{equation}
\mathcal{L}^{v(.)}\varphi(\eta,x,t)=D_{x}\varphi(\eta,x,t).g(x,v(t))-A_{x}\varphi(\eta,x,t)+\dfrac{1}{2}\eta^{2}|h(x)|^{2}\dfrac{\partial^{2}\varphi(\eta,x,t)}{\partial\eta^{2}}\label{eq:6-5}
\end{equation}

\begin{equation}
\mathcal{M}\varphi(\eta,x,t)=\eta\dfrac{\partial\varphi(\eta,x,t)}{\partial\eta}h(x)\label{eq:6-6}
\end{equation}
Then the DMZ equation is 

\begin{equation}
d\rho(\varphi)(t)=\rho(\dfrac{\partial\varphi}{\partial t}+\mathcal{L}^{v(.)}\varphi)(t)dt+\rho(\mathcal{M}\varphi)(t).dz(t)\label{eq:6-7}
\end{equation}

\[
\rho(\varphi)(0)=E\varphi(\eta_{0},x_{0},0)
\]
 In the sequel we assume the existence of a density $p(\eta,x,t)$
which is the joint conditional probability density of $\eta(t),x(t)$
given $\mathcal{Z}^{t}$ . It is defined by 

\begin{equation}
\rho(\varphi)(t)=\int p(\eta,x,t)\varphi(\eta,x,t)d\eta dx\label{eq:6-8}
\end{equation}
 The conditional probability density is the solution of the stochastic
P.D.E. 

\begin{equation}
dp+[A_{x}^{*}p(\eta,x,t)+\label{eq:6-9}
\end{equation}

\[
+\text{div}(g(x,v(t))p(\eta,x,t))-\dfrac{1}{2}|h(x)|^{2}\dfrac{\partial^{2}}{\partial\eta^{2}}(\eta^{2}p(\eta,x,t))\,]dt=-h(x)\dfrac{\partial}{\partial\eta}(\eta p(\eta,x,t)).dz(t)
\]
\[
p(\eta,x,0)=p_{0}(\eta,x)
\]
 It is easy to check that 

\begin{equation}
q(x,t)=\int\eta p(\eta,x,t)d\eta\label{eq:6-10}
\end{equation}
is the solution of Zakai equation (\ref{eq:2-9}) , provided that
the initial condition

\begin{equation}
q_{0}(x)=\int\eta p_{0}(\eta,x)d\eta\label{eq:6-11}
\end{equation}
It is then clear , that although $p(\eta,x,t)$ is indeed a probability
density , $q(x,t)$ is not . Conversely, if we start with $q_{0}(x)$
and want to solve Zakai equation (\ref{eq:2-9}) , we can use (\ref{eq:6-10})
by looking for $p(\eta,x,t)$ solution of the DMZ equation (\ref{eq:6-9})
. We need to take the initial condition 

\begin{equation}
p_{0}(\eta,x)=\delta(\eta-\int q_{0}(\xi)d\xi)\otimes\dfrac{q_{0}(x)}{\int q_{0}(\xi)d\xi}\label{eq:6-12}
\end{equation}
This is not a probability density, so we need to use the weak formulation,
to proceed. 

We get some kind of interesting quandary. Using the set up \cite{BCFP}
we can use probability measures, the Wasserstein topology and the
lifting method of P.L. Lions, but the price to pay is to increase
the dimension by 1, with a nonlinearity. If we stay with the traditional
set up, we have to work with unnormalized probability densities. If
we can work with densities, it is not a serious drawback, but otherwise
we have to find an alternative to the Wasserstein space and the lifting
procedure, and it is not clear how to proceed. We can ,of course,
consider Kushner equation, instead of Zakai equation, whose solution
is a probability. But Kushner equation is nonlinear, conversely to
Zakai equation. 

\subsection{BELLMAN EQUATION }

We can extend the comparison at the level of Bellman equation. We
consider $p^{v(.)}(\eta,x,s)$ , $s\geq t$ solution of 

\begin{equation}
dp+[A_{x}^{*}p(\eta,x,s)+\label{eq:6-13}
\end{equation}

\[
+\text{div}(g(x,v(s))p(\eta,x,s))-\dfrac{1}{2}|h(x)|^{2}\dfrac{\partial^{2}}{\partial\eta^{2}}(\eta^{2}p(\eta,x,s))\,]ds=-h(x)\dfrac{\partial}{\partial\eta}(\eta p(\eta,x,s)).dz(s)
\]
\[
p(\eta,x,t)=p(\eta,x)
\]
where $p(\eta,x)$ is a probability density , denoted in the sequel
by $p$. We can then define the payoff $J_{p,t}(v(.))$ by 

\begin{equation}
J_{p,t}(v(.))=E[\int_{t}^{T}\int p^{v(.)}(\eta,x,s)\eta f(x,v(s))d\eta dxds+\int p^{v(.)}(\eta,x,T)\eta f_{T}(x)d\eta dx]\label{eq:6-14}
\end{equation}
and we define the value function 

\begin{equation}
\Psi(p,t)=\inf_{v(.)}J_{p,t}(v(.))\label{eq:6-15}
\end{equation}
We can write Bellman equation corresponding to this problem. Indeed, 

\begin{equation}
\dfrac{\partial\Psi}{\partial t}+\int[-A_{x}\dfrac{\partial\Psi}{\partial p}(p,t)(\eta,x)+\label{eq:6-16}
\end{equation}
\[
+\dfrac{1}{2}|h(x)|^{2}\eta^{2}\dfrac{\partial^{2}}{\partial\eta^{2}}\dfrac{\partial\Psi}{\partial p}(p,t)(\eta,x)]p(\eta,x)d\eta dx+
\]
\[
+\dfrac{1}{2}\int\int\eta\tilde{\eta}\dfrac{\partial^{2}}{\partial\eta\partial\tilde{\eta}}\dfrac{\partial^{2}\Psi}{\partial p^{2}}(p,t)(\eta,x;\tilde{\eta},\tilde{x})h(x).h(\tilde{x})p(\eta,x)p(\tilde{\eta},\tilde{x})d\eta d\tilde{\eta}dxd\tilde{x}\,+
\]
\[
+\inf_{v}\int[\eta f(x,v)+D_{x}\dfrac{\partial\Psi}{\partial p}(p,t)(\eta,x).g(x,v)]p(\eta,x)d\eta dx=0
\]
\[
\Psi(p,T)=\int\eta f_{T}(x)p(\eta,x)d\eta dx
\]
 If we compare with the Bellman equation (\ref{eq:3-7}) rewritten
with the current notation ( with argument an unormalized probability)
we obtain 

\begin{equation}
\dfrac{\partial\Phi}{\partial t}+\int[-A_{x}\dfrac{\partial\Phi}{\partial q}(q,t)(x)q(x)dx+\label{eq:6-17}
\end{equation}
\[
+\dfrac{1}{2}\int\int\dfrac{\partial^{2}\Phi}{\partial q^{2}}(q,t)(x;\tilde{x})h(x).h(\tilde{x})q(x)q(\tilde{x})dxd\tilde{x}\,+
\]
\[
+\inf_{v}\int[f(x,v)+D_{x}\dfrac{\partial\Phi}{\partial q}(q,t)(x).g(x,v)]q(x)dx=0
\]
\[
\Phi(q,T)=\int f_{T}(x)q(x)dx
\]
Equations (\ref{eq:6-16}) and (\ref{eq:6-17}) are linked by the
formulas 

\begin{equation}
\Psi(p,t)=\Phi(\int\eta p(\eta,.)d\eta,t)\label{eq:6-18}
\end{equation}

\begin{equation}
\Phi(q,t)=\Psi(\delta(.-\int q(\xi)d\xi)\otimes\dfrac{q(.)}{\int q(\xi)d\xi},t)\label{eq:6-19}
\end{equation}

\subsection{THE LINEAR CASE}

If we go back to the linear case ( \ref{eq:4-3}) , we get 

\begin{equation}
dx=(Fx+Gv)dt+\sigma dw\label{eq:4-3-1}
\end{equation}
\[
d\eta=\eta.Hx.dz
\]
 therefore , in the set up \cite{BCFP}, we still have a nonlinear
system. Therefore, we cannot use the linear case of \cite{BCFP}.
This explains why our formulas are completely different. The fact
that we have an explicit solution of the system of HJB-FP equations
does not imply that we have an explicit solution of Bellman equation.
This is consistent with the spirit of the method of characteristics.


\begin{thebibliography}{1}
\bibitem[1]{BCFP}E. Bandini, A. Corso, M. Fuhrman, H. Pham, Randomized
filtering and Bellman equation in Wasserstein space for part incraseial
observation control problem, arXiv, sept. 2016

\bibitem[2]{ABE} A. Bensoussan, \emph{Stochastic Control of Partially
Observable Systems, }Cambridge University Press, 1992

\bibitem[3]{CDLL} P. Cardialaguet, F. Delarue, J.M. Lasry, P.L. Lions,
The Master Equation and the Convergence Problem in Mean Field Games,
HAL archives, September , 2015

\bibitem[4]{ARM} A. Makowsky, Filtering Formulae for Partially Obeserved
Linear Systems with Non-Gaussian Initial Conditions, Stochastics,
16, 1-24
\end{thebibliography}
\end{document}